\documentclass[12pt,a4paper]{article}

\usepackage{amssymb}
\usepackage{amsmath}
\usepackage{indentfirst}
\usepackage{t1enc}
\usepackage{array}
\usepackage{graphpap}
\usepackage[mathscr]{eucal}
\usepackage{enumerate}
\usepackage{tabularx}
\usepackage{graphicx}
\usepackage[colorlinks=true,linkcolor=blue,citecolor=blue,
  urlcolor=blue]{hyperref}
\usepackage[hyphenbreaks]{breakurl}
\usepackage{amsthm}

\usepackage{tikz} 
\usepackage{pgfplots} 
\usetikzlibrary{shapes,arrows}

\usepackage[latin2,utf8]{inputenc}

\DeclareMathOperator{\Taxicab}{Taxicab}

\begin{document}

\theoremstyle{plain}
\newtheorem{thm}{Theorem}[section]
\newtheorem{cor}[thm]{Corollary}
\newtheorem{lem}[thm]{Lemma}
\newtheorem{dfn}[thm]{Definition}
\newtheorem{rmk}[thm]{Remark}
\newtheorem{cns}[thm]{Construction}
\newtheorem{exm}[thm]{Example}
\newtheorem{prs}[thm]{Proposition}
\newtheorem{prob}[thm]{Problem} 
\newtheorem{ntn}[thm]{Notation}
\newtheorem{conj}[thm]{Conjecture}
\newtheorem{prop}[thm]{Proposition}
\newtheorem{fel}[thm]{Exercise}

\newcommand{\gc}{\operatorname{gcd}}
\newcommand{\mc}{\mathcal}
\newcommand{\mr}{\mathscr}
\newcommand{\ab}[1]{\left\vert{#1}\right\vert}
\newcommand{\zj}[1]{\left({#1}\right)}
\newcommand{\norma}{N_{\mathbb F_{q^n}/\mathbb F_q}}
\newcommand{\normawp}{N_{\mathbb F_{p^n}/\mathbb F_p}}
\newcommand{\normfield}{\mc N_{\mathbb F_{q^n}/\mathbb F_q} (f)}

\renewcommand{\baselinestretch}{1.38}
\renewcommand\thefootnote{\relax}

\newcommand{\bs}{\bigskip}
\newcommand{\bi}{\bigskip\noindent}
\newcommand{\red}[1]{\textcolor{red}{#1}}
\newcommand{\lb}[1]{\label{#1}}
\allowdisplaybreaks[1] 

\normalsize

\sloppy


\title{Ramanujan, the taxicab problem for polynomials, and the abc-conjecture}

\author{Katalin Gyarmati}
\date{}

\footnotetext{
\indent Research supported by the Hungarian National Research Development and 
Innovation Fund KKP133819.}

\maketitle

\begin{abstract}
  Starting with \mbox{Ramanujan}'s famous taxicab problem, we can
  study the solvability of the equations $p^n+q^n=r^n+s^n$
  and, more generally, $p_1^{k_1}+\dots+p_m^{k_m}=0$ among polynomials.
\end{abstract}  

\section{Introduction}

The taxicab problem is one of the best-known anecdotes about \mbox{Ramanujan}.
Hardy
\cite{Hardy1}, \cite{Hardy2} wrote the following:
``I remember once going to see him when he was lying ill at Putney.
I had ridden in taxi cab number 1729 and remarked that the number seemed
to me rather a dull one, and that I hoped it was not an unfavorable omen.
‘No,’ he replied, ‘it is a very interesting number; it is the smallest number
expressible as the sum of two cubes in two different ways.’ ''

\bs However, Turán \cite{TP} provides another version of this anecdote,
according to which Hardy and \mbox{Ramanujan}
traveled together in the taxi, and when they got out,
Hardy forgot to retrieve his briefcase,
containing  important manuscripts.
Hardy was very distraught, but \mbox{Ramanujan} reassured him that
he remembered the taxicab number, as it was a very interesting number;
namely, it was the smallest number that could be written
in two different ways as the sum of two cubes.

\bs Regardless of which version of the anecdote is true, 
there is no doubt that in
the so-called ``Lost Notebook'' \cite{Rama}, \mbox{Ramanujan} found an
infinite number of examples of triples $x,y,z\in\mathbb Z^+$
for which
\begin{equation}
x^3+y^3=z^3\pm 1. 
\lb{ref01}
\end{equation}
Perhaps \mbox{Ramanujan} was unaware of Euler's proof of
the Fermat conjecture in the case of exponent $n=3$ and that he 
sometimes wanted
to prove the conjecture, while other times, he wanted to disprove it.
He almost succeeded in the latter endeavor: as in his triples,
the sum of the first two cubes is ``almost'' a cube.

\bs Hirschorn \cite{Hirs1}, \cite{Hirs2}, \cite{Hirs3}, as well as
many others,
tried to reconstruct the methods
\mbox{Ramanujan} originally used to find these triples. However, in this paper,
I use an important earlier result of Lehmer \cite{Lehmer}—one which
\mbox{Ramanujan} probably did not think about but which can be used to
find an infinite number of parametric solutions to equation
\eqref{ref01}.
According to this result, the equation
\[
  x^3+y^3=z^3+1
\]
has infinitely many solutions. For example:
\begin{align*}
  x&=    9t^4,\ \ y=    1+9t^3,\ \  z=    9t^4+3t.          
\end{align*}  
Similarly, the equation
\[
x^3+y^3=z^3-1
\]
has infinitely many solutions, namely:
\begin{align*}
  x= 9t^4-3t,\ \             
  y= 9t^3-1,\ \             
  z= 9t^4.              
\end{align*} 
(Lehmer also found additional solutions by studying Pell equations.)

\bs There is a related question: 
For $n\ge 4$, are there ``different'' polynomials $p,q,r,s$ such that
\[
p^n+q^n=r^n+s^n
\]
holds? Although it would be useful
to provide some concrete examples of such polynomials
(if such polynomials exist).
I could not solve this problem in general. Nevertheless,
I can state the following conjecture:

\begin{conj}\lb{conj1.1}
  There do not exist polynomials $p,q,r,s\in\mathbb Z[x]$
  such that 
  \[
  p^n+q^n=r^n+s^n\ne 0
  \]
  with an integer $n\ge 5$
  and $\{p^n,q^n\}\ne\{r^n,s^n\}$.
\end{conj}

Here, I note that Granville and Tucker \cite{GrTuck}
wrote an excellent expository paper on a related problem,
namely the extension of Fermat's conjecture to polynomials.

\section{A few extra conditions}

When writing a paper, the first thing to remember is that
if we cannot prove a theorem in its original form, a few conditions
can be applied
to get a slightly weaker (and easier to prove) form.
That’s just what I did, and in doing so, I was able to prove the following:

\begin{thm}\lb{thm2.1}
  If $n\ge 16$, there do not exist polynomials
  $p,q,r,s\in\mathbb C[x]$
  such that the equation
 \begin{equation}
  p^n+q^n=r^n+s^n\ne 0\lb{foegy}
 \end{equation}
 holds with $\{p^n,q^n\}\ne \{r^n,s^n\}$,
 $\max\{\deg p,\deg q,\deg r,\deg s\}\ge 1$ and
  $\gc (p,q,r,s)=1$.
\end{thm}

\bi In special cases, this theorem was already known.
For example,
Newman and Slater \cite[p. 481]{NewmanSlater} proved
the theorem in 1979 using Wronskians in the special case where one of
these polynomials is constant.
Later, in 2004,
Bayat and Teimorii \cite{Bayat} handled a case where the
polynomials in the theorem are pairwise coprime. However,
using a little less Wronskians in Bayat and Teimorii’s proof and adding
a little more gcd makes the proof simpler and more general.
Theorem \ref{thm2.1} follows from a theorem of de Bondt \cite{Bondt}
(see also section Remarks) and from Theorem \ref{thm2.3} of this paper.
Although Theorem \ref{thm2.1} can be obtained as a special case from more general theorems, I believe it is important to provide a simple and direct proof. This direct proof serves as the starting point for the proof of Theorem \ref{thm2.3}.

\bi Theorem \ref{thm2.1} does not apply when all polynomials
are constant. In this case the best result is due to Elkies \cite{Elk},
who only proved that if $2\le A,B,C\le 8,388,608 \in\mathbb N$
and $4\le n\le 2^{23}\in\mathbb N$,
then the equation
\[
A^n+B^n=C^n\pm 1,
\] 
has no solution.

\bi Theorem \ref{thm2.1} is strongly related to the generalized
taxicab problem.

\begin{dfn} Let $\Taxicab (n,k,j)$ denote the smallest positive integer,
  that can be written at least in $j$ ways as the sum of $k$ pieces of
  $n$-th powers.
\end{dfn}

\bi So far, we know very little about generalized taxicab numbers.
According to our theorem, the existence of generalized taxicab numbers
$\Taxicab (n,2,j)$
for $n\ge 16$ and $j\ge 2$
cannot be proved with a simple polynomial construction.
A crucial question is whether Theorem \ref{thm2.1} can be extended to include
more summands and polynomials with different exponents.
Using Wronskians, the following can be proved.
\begin{thm}\lb{thm2.3}
  Suppose that $m\ge 3$ is an integer and for the nonzero
  polynomials $p_1,p_2,\dots,p_m\in\mathbb C[x]$
  and positive integers $k_1,k_2,\dots,k_m\in\mathbb N$
  the following conditions hold:
\begin{align}
\gcd(p_1^{k_1}, p_2^{k_2},\dots,p_m^{k_m}) = 1, \lb{ll01}
\end{align}
\[
p_1^{k_1} + p_2^{k_2} + \ldots + p_m^{k_m} = 0,
\]
and any $m-1$ of the polynomials $p_i^{k_i}$ are linearly
independent. Then,
\begin{align}
  \min \{k_1,k_2,\dots,k_m\}<
  \dfrac{1}{3}\zj{m^3+7m^2-49m+68}.
  \lb{ll02}
\end{align}
\end{thm}

An interesting question is whether condition \eqref{ll01} can be omitted
from the theorem. The answer is negative, due to the following reason:
Let's assume that the exponents $k_1, k_2, \ldots, k_m$ are pairwise coprime
and for the polynomials $f_1,f_2,\dots,f_m$ we have
\[
f_1 + f_2 + \ldots + f_m = 0,
\]
where $f_1, \ldots, f_m$ are not all constant polynomials.

By the Chinese Remainder Theorem, there exist integers $\alpha_i$ such that $k_1 k_2 \ldots k_{i-1} k_{i+1} \ldots k_m | \alpha_i$ and $k_i | \alpha_{i}+1$. Multiplying the equation by $f_1^{\alpha_1} f_2^{\alpha_2} \ldots f_m^{\alpha_m}$ yields:
\[
f_1^{\alpha_1+1} f_2^{\alpha_2} \ldots f_m^{\alpha_m} + \ldots + f_1^{\alpha_1} \ldots f_{m-1}^{\alpha_{m-1}} f_m^{\alpha_m+1} = 0.
\]
Now, define the polynomials $p_j\in\mathbb C[x]$ as follows:
\[
p_j^{k_j} \stackrel{\textup{def}}{=} f_1^{\alpha_1} \ldots f_{j-1}^{\alpha_{j-1}}f_j^{\alpha_j+1} f_{j+1}^{\alpha_{j+1}} \ldots f_m^{\alpha_m}.
\]
This construction gives a sum of the form $p_1^{k_1} + \dots + p_m^{k_m} = 0$, but with $\gcd(p_1^{k_1}, \dots, p_m^{k_m}) \neq 1$. It is clear that if $\min\{k_1, \dots, k_m\}$ is sufficiently large, inequality \eqref{ll02} is not satisfied.

\bigskip
The following can be easily deduced from Theorem \ref{thm2.3}:

\begin{cor}\lb{cor2.4}
  Suppose that $m\ge 3$ is an integer and for the nonzero
  polynomials $p_1,p_2,\dots,p_m\in\mathbb C[x]$
  and positive integer $k$
  the following hold:
\[
p_1^{k} + p_2^{k} + \ldots + p_m^{k} = 0,
\]
and the quotient of any two different polynomials $p_i$ is never
constant. Then,
\begin{align*}
  k<
  \dfrac{1}{3}\zj{m^3+7m^2-49m+68}.
\end{align*}
\end{cor}


\section{The generalized abc-conjecture}

The main tool in the proof of Theorem
\ref{thm2.1} and Theorem \ref{thm2.3} is Mason's theorem,
which has many generalizations (see, e.g., \cite{Br},
\cite{Hu1}, \cite{Hu2}, \cite{Masonfunction}, \cite{Silverman}, \cite{Voloch}
and \cite{Zannier}). However, most of the results rely on function
fields, which I decided to avoid for the sake of clarity.
Fortunately, there is a generalization that uses polynomial
rings. We introduce the following notation: 
For an arbitrary (potentially multivariable) polynomial $p$,
let its unique factorization into irreducible polynomials be
\[
p=p_1^{\alpha_1}p_2^{\alpha_2}\cdots p_r^{\alpha_r},
\]
where the irreducible polynomials $p_i$ are pairwise coprime, and
let $\textup{rad }p$ denote the following polynomial: 
\[
\textup{rad }p\stackrel{\textup{def}}{=}p_1p_2\dots p_r. 
\]
Here $\textup{rad}\ p$ is unique apart from a constant factor.
Moreover, let $r^*(p)$ denote the degree of the polynomial $\textup{rad }p$.
The following lemma is due to Shapiro and Sparer \cite{ShapiroSparer}.

\begin{lem}
  If the (possibly multivariate)
  polynomials $f_1,f_2,\dots,f_m$ over $\mathbb C$ are pairwise coprime,
  not all of them are
  constant and
  \[
  f_1+f_2+\dots+f_m=0,
  \]
  then
  \[
  \max_{1\le i\le m}\operatorname{deg} f_i
  \le (m-2)\zj{r^*(f_1f_2\dots f_m)-1}.
  \]
\end{lem}

The following theorem follows easily from this lemma. However,
I have omitted the proof since there is not enough space to
provide it here.

\begin{thm}\lb{thm3.5}
  If the polynomials $f_1,f_2,\dots f_{k},g_1,g_2,\dots,g_k \in\mathbb C[x]$
  are pairwise coprime, and at least one of them is not constant,
  then for $n\ge 4k(k-1)$ the equation
\[
f_1^n+f_2^n+\dots +f_k^n=g_1^n+g_2^n+\dots +g_k^n
\]
never holds.
\end{thm}

In fact, this theorem is related to Corollary \ref{cor2.4},
with the difference that the pairwise
coprimality condition is not required there, while
using Corollary \ref{thm2.3}, the statement of this theorem would
only follow in case of
$n\gg k^3$.

\bs Motivated by Ruzsa's talk \cite{R}, I propose the following proposition.

\begin{prop}
Let $\mathcal{F}$ be the following family of triples of integers:
\[
  \mathcal{F}=\{(a,b,c):\ a=(2^{\alpha}-1)^2(2^{\alpha+2}-1),
  \ b=(3\cdot2^{\alpha}-1)^2,\ c=2^{3\alpha+2})\}.
\]
Then, the elements of the set $\mathcal{F}$ satisfy the $abc$-conjecture,
namely, for all $\varepsilon>0$ there exists a
constant $K_{\varepsilon}$ such that if $(a,b,c)\in\mathcal{F}$, then $a+b=c$ and

\[
c<K_{\varepsilon} (\textup{rad }(abc))^{1+\varepsilon}.
\]
\end{prop}

\bs The statements of this section serve as useful illustrative examples of the
$abc$-conjecture for university students.
It is hoped that propositions and exercises of this type
could effectively enhance students' problem-solving skills for future
research related to this subject.

\section{Almost disproving a conjecture}

I almost managed to disprove my Conjecture 1.1
as follows. If we replace $\mathbb Z$ in Conjecture 1.1 with another ring,
say $\mathbb Q(\sqrt[4]{2}+i)$ (which is not only a ring, but also a field),
and use $n=4$ in place of $n\ge 5$,
then the conjecture does not hold. More precisely:

\begin{exm}\lb{thm4.1}
  If $\varepsilon=e^{i\pi/4}=\dfrac{\sqrt{2}}{2}+\dfrac{\sqrt{2}}{2}i$, then
  for the four polynomials
  \[
  p(x)=x^4+1, \ \ q(x)=\varepsilon 8^{1/4}x^3,\ \ r(x)=x^4-1,\ \ s(x)=8^{1/4}x,
  \]
  of degree $\le 4$, we have
  \[
  p^4+q^4=r^4+s^4\ne 0
  \]
  and $p,q,r,s\in\mathbb Q\zj{\sqrt[4]{2}+i}$.
\end{exm}

Studying this example helped to better understand the scope and limitations
of Conjecture \ref{conj1.1}.

The novelty of Example \ref{thm4.1} cannot be underestimated by the fact
that Euler already proved a similar (but not the same) result.
This is because Euler
used polynomials of degrees $6$ and $7$. Namely, for the polynomials
\begin{align*}
p(x)&=x^7+x^5-2 x^3+3 x^2+x\\
q(x)&=x^6-3 x^5-2 x^4+x^2+1\\
r(x)&=x^7+x^5-2 x^3-3 x^2+x\\
s(x)&=x^6+3 x^5-2 x^4+x^2+1
\end{align*}
the equation
\[
p^4+q^4=r^4+s^4
\]
holds.

\bigskip

Finally, to show that the polynomials $p,q,r$ and $s$
from Example \ref{thm4.1} are indeed in $\mathbb Q\zj{\sqrt[4]{2}+i}$,
we need to prove that
$8^{1/4}$ and $\varepsilon$, are elements of $\mathbb Q\zj{\sqrt[4]{2}+i}$.
To confirm this, we will follow the method used in \cite{bovites}.
Let $\alpha=\sqrt[4]{2}+i$. Then,
\begin{align*}
  (\alpha-i)^4&=2
\end{align*}
and
\begin{align*}
  \alpha^4-4\alpha^3i-6\alpha^2+4\alpha i+1 &=2.
\end{align*}
So,
\begin{align*}
  i&=\dfrac{\alpha^4-6\alpha^2-1}{4\alpha^3-4\alpha}.
\end{align*}
By this, $i\in\mathbb Q(\alpha)=\mathbb Q\zj{\sqrt[4]{2}+i}$.
But then $\sqrt[4]{2}=\alpha-i\in\mathbb Q(\alpha)$.
That is, if we consider the cube of $\sqrt[4]{2}$,
we get $8^{1/4}\in\mathbb Q(\alpha)$. Furthermore, if we consider the
square of $\sqrt[4]{2}$ and divide it by 2, we get
$\dfrac{\sqrt{2}}{2}\in\mathbb Q(\alpha)$. This implies
that $\varepsilon=\dfrac{\sqrt{2}}{2}+\dfrac{\sqrt{2}}{2}i\in\mathbb Q(\alpha) =\mathbb Q\zj{\sqrt[4]{2}+i}$. 

\section{Proofs}

\bi\textbf{Proof of Theorem  \ref{thm2.1}.}
The main tool of the proof is the Mason-Stothers theorem, initially called Mason's theorem, which was first proved by Stothers \cite{Sto} and a little later by Mason \cite{Masonfunction}. Snyder gave a very elegant proof in \cite{Syd}.

\begin{lem}[Mason]\lb{lem5.2}
  Let $F$ be an algebraically closed field.
  Suppose that
  $f,g,h\in F[x]$ are polynomials with $\gc (f,g,h)=1$, $f+g=h$
  and that there is a
  nonzero between their derivatives. Then,
  \[
  \max\{\deg f,\deg g,\deg h\}\le
  r^*(fgh)-1.
  \] 
\end{lem}  

In 1994, Wiles proved the famous Fermat conjecture for
$n\ge 3$.
Paradoxically, the fact that there is no solution to $a^n+b^n=c^n$
for coprime polynomials
different from a constant has been known since the XIX century, and
the first proof used deep algebraic tools. A more modern approach shows
that this result easily follows from Mason's theorem:

\begin{lem}\lb{lem5.2}
  There are no polynomials $a,b,c\in\mathbb C[x]$ and an integer
  $n \ge 3$
  for which 
  $\gc (a,b,c)=1$ and
  \[
  a^n+b^n=c^n
  \]
  holds. 
\end{lem}
(For a simplified proof using Mason's theorem, see, e.g., \cite{Lange}.)

\bigskip

Let us return to the proof of Theorem \ref{thm2.1}.
Without loss of generality, we may assume that among the
polynomials $p,q,r,s$, the one with the highest degree is $s$. Let 
\[
k\stackrel{\textup{def}}{=}\deg s=\max\{\deg p,\deg q,\deg r,\deg s\}.
\]
If one of the polynomials $p,q,r$ is a constant multiple of the
polynomial $s$, then Theorem \ref{thm2.1} follows from Lemma \ref{lem5.2}.
For example, let $p=cs$, where $c$ is a constant. If $c^n=1$, then $p^n=s^n$,
and thus, $q^n=r^n$ also holds, which contradicts the conditions of the
theorem. If $c^n\ne 1$, then the equation
\[
p^n+q^n=r^n+s^n
\]
can be rearranged as
\[
(c^n-1)s^n+q^n=r^n.
\]
Now $(c^n-1)^{1/n}$ is a complex number,
so $s_0\stackrel{\textrm{def}}{=}(c^n-1)^{1/n }s\in\mathbb C[x]$
also holds, and thus
\[
s_0^n+q^n=r^n.
\]
However, this contradicts Lemma \ref{lem5.2}.
The cases $q=cs$ and $r=cs$, where $c$ is a constant, can be handled
similarly.

\bi We may assume that each of $p,q,r$ is not a constant multiple of $s$.
That is, none of the rational functions $\dfrac{p}{s}$, $\dfrac{r}{s}$,
$\dfrac{q}{s}$ is constant. Thus, the derivatives of these
rational functions are not zero. Then, by \eqref{foegy}
\[
\zj{\dfrac{p}{s}}^n+\zj{\dfrac{q}{s}}^n
= \zj{\dfrac{r}{s}}^n+1. 
\]
Differentiating the function equation yields:
\[
n\zj{\dfrac{p}{s}}^{n-1}
\dfrac{p's-ps'}{s^2}
+n\zj{\dfrac{q}{s}}^{n-1}
\dfrac{q's-pq'}{s^2}=
n\zj{\dfrac{r}{s}}^{n-1}
\dfrac{r's-rs'}{s^2}.
\]
Multiplying the equation by $\dfrac{1}{n}s^{n+1}$, we get:
\begin{align}
p^{n-1}(p's-ps')
+q^{n-1}(q's-qs')
=r^{n-1}(r's-rs').\lb{r2} 
\end{align}
The polynomials $p's-ps'$, $q's-qs'$, $r's-rs'$ are
the so-called Wronskians. (Interested readers can read more about
Wronskians on the related Wikipedia page \cite{Wrons}.)
We want to use Mason's theorem in equation \eqref{r2}.
At first glance, $f=p^{n-1}(p's-ps')$, $g=q^{n-1}(q's-qs')$,
$h=r^{n-1}(r's-pr')$ seem to be a good choice, but Lemma \ref{lem5.2}
requires that the condition $\gc{f,g,h}$ must be satisfied.
Thus, we need to slightly modify the definitions of $f,g,h$. Let
\[
d\stackrel{\textup{def}}{=}
\gc \{p^{n-1}(p's-ps'),q^{n-1}(q's-qs'),
r^{n-1}(p's-rs')\}
\]
Then, by \eqref{r2},
\begin{align*}
\dfrac{p^{n-1}(p's-ps')}{d}
+\dfrac{q^{n-1}(q's-qs')}{d}
=\dfrac{r^{n-1}(r's-rs')}{d}. 
\end{align*}
Next, we want to use Lemma \ref{lem5.2} with
\begin{align}
f&=\dfrac{p^{n-1}(p's-ps')}{d},\notag\\
g&=\dfrac{q^{n-1}(q's-qs')}{d},\notag\\
h&=\dfrac{r^{n-1}(r's-rs')}{d}.\lb{r3a}
\end{align}
To do this, we first note that the
polynomials $p's-ps'$, $q's-qs'$, $r's-rs'$ are not identically zero, since,
for example,  if $p's-ps'=0$, then $\zj{\dfrac{p}{s}}'=0$,
and thus, $\dfrac{p}{s}$ is a constant. We excluded this case
at the beginning of the proof. Next, we will need the following:
\begin{lem}\lb{lem5.3}
If $k=\max\{\deg p,\deg q,\deg r,\deg s\}=\deg s$, 
then
\[
\deg d\le  6k-3
\]
\end{lem}

First we note that
\begin{align}
\gc\{p,q,r\}=1,\lb{r4}
\end{align}
since
$\gc\{p,q,r\}^n \mid p^n+q^n-r^n=s^n$, and thus
$\gc\{p,q,r\}\mid s$. So
$\gc\{p,q,r\}\mid \gc\{p,q,r,s\}=1$, from which
$\gc\{p,q,r\}=1$ follows.

\bi Let us write $d$ in the form
\[
d = u_1^{\alpha_1} \dots u_s^{\alpha_s},
\]
where $u_1,\dots,u_s \in \mathbb{C}[X]$ are pairwise coprime irreducible polynomials.
Then,
\[
u_i^{\alpha_i} | p^{n-1}(p's-ps'), q^{n-1}(q's-qs'), r^{n-1}(r's-rs').
\]
Since $(p,q,r)=1$, one of the following must hold: $(u_i^{\alpha_i}, p^{n-1})=1$ or $(u_i^{\alpha_i}, q^{n-1})=1$ or $(u_i^{\alpha_i}, r^{n-1})=1$. In all three cases, by \eqref{r2},
\[
u_i^{\alpha_i} | (p's-ps)(q's-qs')(r's-rs').
\]
Since this holds for each factor $u_i^{\alpha_i}$ of $d$ and the polynomials
$u_i$ are pairwise coprime, we get
\[
d | (p's-ps)(q's-qs')(r's-rs').
\]
Since $p's-ps, q's-qs$ and $r's-rs$ are nonzero polynomials, we have
\[
\deg d \leq \deg (p's-ps) + \deg (q's-qs) + \deg (r's-rs).
\]
By $\deg p, \deg q, \deg r \leq k$, we get
\[
\deg d \leq (2k-1) + (2k-1) + (2k-1) = 6k-3,
\]
and this completes the proof of the lemma.

\bi If $p^n+q^n=r^n+s^n$, where $s$ has the maximal degree,
and $\deg s=k$, then among
$p,q,$ and $r$, at least one polynomial has a degree of $k$
(otherwise, the degree of $p^n+q^n-r^n=s^n$ would be less than $kn$).
Using this,
\eqref{r3a}, Lemma \ref{lem5.3} 
and $\deg p's-ps',\deg q's-qs',\deg r's-rs'\ge 1$, we find that
at least one of the polynomials $f,g,h$ defined in \eqref{r3a}
has degree $\ge (n-1)k+1-(6k-3)=(n-7)k+4\ge 4$.
Consequently, $f'=g'=h'=0$ cannot hold. Moreover, we have also proved that
\[
\max\{\deg f,\deg g,\deg h\}\ge (n-7)k+4.
\]
\bi Using Lemma \ref{lem5.2}, we get
\[
\max \{\deg f,\deg g,\deg h\}\le r^*(fgh)-1.
\]
Thus,
\[
(n-7)k+4\le r^*(fgh)-1.
\]
But
\begin{align*}
  r^*(fgh)&=\deg\textup{rad }(fgh)\\
  &\le \deg \textup{rad } (p^{n-1}(p's-ps')q^{n-1}(q's-qs')r^{n-1}(r's-rs'))\\
  &=  \deg \textup{rad } (pqr(p's-ps')(q's-qs')(r's-rs'))\\
  &\le k+k+k+(2k-1)+(2k-1)+(2k-1)=9k-3.
\end{align*}
So,
\begin{align*}
  (n-7)k+4&\le 9k-4\\
  (n-7)k&<9k\\
  n&<16, 
\end{align*}
which contradicts the conditions of the theorem. 
This completes the proof of Theorem \ref{thm2.1} 

\bigskip\noindent\textbf{Proof of Theorem \ref{thm2.3}.}
We will use the following notation.
\begin{dfn}\lb{dfn5.6}
Let
\begin{align*}
  a_i &\stackrel{\textup{def}}{=} p_i^{k_i},\\
  K  &\stackrel{\textup{def}}{=}\min \{k_1,\dots,k_m\},\\
  T  &\stackrel{\textup{def}}{=}
       \max\{\textup{deg }p_1,\dots,\textup{deg }p_m\},\\
  \intertext{and}
  t_i&\stackrel{\textup{def}}{=}\textup{deg }p_i.
\end{align*}
Without loss of generality, we can assume that
\[
T=\textup{deg }p_1=t_1.
\]
Finally, for $1\le i\le m-3$ and $0\le n\le k_i$, we define
$b_{i,n}\in\mathbb C[x]$ by
\begin{align}
  a_i^{(n)}=p_i^{k_i-m+3}b_{i,n},\lb{bin}
\end{align}
where $a_i^{(n)}$ is the $n$-th derivative of $a_i=p_i^{k_i}$. (In this
definition, the exponent of $p_i$ is $k_i-m+3$, and not $k_i-n+3$.)
\end{dfn}  

\bigskip
We will proceed with a proof by contradiction. Suppose that
\begin{align}
K = \min \{k_1,\dots,k_m\} \ge \dfrac{1}{3}\zj{m^3+7m^2-49m+68}.\lb{KK1}
\end{align}  

Let us now return to the proof of the theorem. Then,
\[
p_1^{k_1} + \dots + p_m^{k_m} = 0,
\]
which can be written as
\[
a_1 + \dots + a_m = 0.
\]
Taking the $n$-th derivative of this equation:
\[
a_1^{(n)} + \dots + a_m^{(n)} = 0.
\]
Consider the following Wronskian determinant:
\begin{align}
W \stackrel{\text{def}}{=} \begin{vmatrix}
a_1+\dots+a_{m}
& a_1' +\dots +a_{m}'&
 \dots & a_1^{(m-3)}+\dots +  a_{m}^{(m-3)}  \\
a_4 & a_4' & \dots & a_4^{(m-3)} \\
a_5 & a_5' & \dots & a_5^{(m-3)} \\
\vdots & \vdots & \ddots & \vdots \\
a_{m} & a_{m}' & \dots & a_m^{(m-3)} 
\end{vmatrix}=0.\lb{wr1}
\end{align}

\bigskip
By subtracting the sum of the other rows from the first row of $W$, we get:
\begin{align}
W&= \begin{vmatrix}
a_1+a_2+a_3 & a_1'+a_2'+a_3' & \dots & a_1^{(m-3)}+a_2^{(m-3)}+a_3^{(m-3)} \\
a_4 & a_4' & \dots & a_4^{(m-3)} \\
a_5 & a_5' & \dots & a_5^{(m-3)} \\
\vdots & \vdots & \ddots & \vdots \\
a_m & a_m' & \dots & a_m^{(m-3)}
\end{vmatrix}\notag\\
&= \footnotesize\begin{vmatrix}
a_1 & a_1' & \dots & a_1^{(m-3)} \\
a_4 & a_4' & \dots & a_4^{(m-3)} \\
a_5 & a_5' & \dots & a_5^{(m-3)} \\
\vdots & \vdots & \ddots & \vdots \\
a_m & a_m' & \dots & a_m^{(m-3)}
\end{vmatrix} + \begin{vmatrix}
a_2 & a_2' & \dots & a_2^{(m-3)} \\
a_4 & a_4' & \dots & a_4^{(m-3)} \\
a_5 & a_5' & \dots & a_5^{(m-3)} \\
\vdots & \vdots & \ddots & \vdots \\
a_m & a_m' & \dots & a_m^{(m-3)}
\end{vmatrix} + \begin{vmatrix}
a_3 & a_3' & \dots & a_3^{(m-3)} \\
a_4 & a_4' & \dots & a_4^{(m-3)} \\
a_5 & a_5' & \dots & a_5^{(m-3)} \\
\vdots & \vdots & \ddots & \vdots \\
a_m & a_m' & \dots & a_m^{(m-3)}
                \end{vmatrix}.\lb{wr2}  
\end{align}                     
It is important to note that none of the three determinants on the right side are $0$. The $i$-th determinant $(1\le i\le 3)$ is the Wronskian of the polynomials $a_i, a_4, \dots, a_m$. 
We know that the Wronskian of analytic functions is $0$ if and only if they are linearly dependent (see e.g., \cite[pp. 91-92]{Boc1}, \cite[Theorem II]{Boc2}, \cite[Chap. 3, \S 7]{Hur} and
\cite[Theorem 3]{Kru}). By the conditions of the theorem, the
polynomials $a_i=p_i^{k_i}, a_2=p_2^{k_2}, \dots, a_m=p_m^{k_m}$ are linearly independent for any $1\le i\le 3$. Consequently, none of the three determinants
above are $0$ (since polynomials are analytic functions).

\bigskip By Definition \ref{dfn5.6}
\begin{align}
  \textup{deg }b_{i,n} &= \textup{deg }a_i^{(n)}
                       -\textup{deg }p_i^{k_i-m+3}\notag\\
                       &=(k_it_i-n)-(k_i-m+3)t_i\notag\\
                       &=(m-3)t_i-n.\lb{sr0} 
\end{align}
Using again Definition \ref{dfn5.6}, \eqref{wr1} and \eqref{wr2}
we get
\begin{align}
  W&=p_1^{k_1-m+3}p_4^{k_4-m+3}\dots p_m^{k_m-m+3}
    \begin{vmatrix}
b_{1,0} & b_{1,1} & \dots & b_{1,m-3} \\
b_{4,0} & b_{4,1} & \dots & b_{4,m-3} \\
\vdots & \vdots & \ddots & \vdots \\
b_{m,0} & b_{m,1} & \dots & b_{m,m-3}
                 \end{vmatrix}\notag\\
&+ p_2^{k_2-m+3}p_4^{k_4-m+3}\dots p_m^{k_m-m+3}
   \begin{vmatrix}
b_{2,0} & b_{2,1} & \dots & b_{2,m-3} \\
b_{4,0} & b_{4,1} & \dots & b_{4,m-3} \\
\vdots & \vdots & \ddots & \vdots \\
b_{m,0} & b_{m,1} & \dots & b_{m,m-3}
    \end{vmatrix}\notag\\
& + p_3^{k_3-m+3}p_4^{k_4-m+3}\dots p_m^{k_m-m+3}
   \begin{vmatrix}
b_{3,0} & b_{3,1} & \dots & b_{3,m-3} \\
b_{4,0} & b_{4,1} & \dots & b_{4,m-3} \\
\vdots & \vdots & \ddots & \vdots \\
b_{m,0} & b_{m,1} & \dots & b_{m,m-3}
    \end{vmatrix}=0.\lb{zhzh}                     
\end{align}
Let
\begin{align}
  \widetilde{w}\stackrel{\text{def}}{=}
  \gcd \zj{p_1^{k_1-m+3},p_2^{k_2-m+3},p_3^{k_3-m+2}}.
\end{align}
The degree of the polynomial $\widetilde{w}$ plays a prominent role in the proof.
\begin{lem}\lb{lemw} If the conditions of the theorem hold, then
 \[
   \deg \widetilde{w}\le \dfrac{1}{3}\zj{m^3-11m^2+38m-40}T.
 \]
\end{lem}

\bigskip We will prove Lemma \ref{lemw} at the end of the proof of the theorem.

Dividing \eqref{zhzh} by $\widetilde{w}p_4^{k_4-m+3}\dots p_m^{k_m-m+3}$ we get
\begin{align}
0&=\dfrac{p_1^{k_1-m+3}}{\widetilde{w}}
    \begin{vmatrix}
b_{1,0} & b_{1,1} & \dots & b_{1,m-3} \\
b_{4,0} & b_{4,1} & \dots & b_{4,m-3} \\
\vdots & \vdots & \ddots & \vdots \\
b_{m,0} & b_{m,1} & \dots & b_{m,m-3}
                 \end{vmatrix}
+ \dfrac{p_2^{k_2-m+3}}{\widetilde{w}}
   \begin{vmatrix}
b_{2,0} & b_{2,1} & \dots & b_{2,m-3} \\
b_{4,0} & b_{4,1} & \dots & b_{4,m-3} \\
\vdots & \vdots & \ddots & \vdots \\
b_{m,0} & b_{m,1} & \dots & b_{m,m-3}
    \end{vmatrix}\notag\\
& + \dfrac{p_3^{k_3-m+3}}{\widetilde{w}}
   \begin{vmatrix}
b_{3,0} & b_{3,1} & \dots & b_{3,m-3} \\
b_{4,0} & b_{4,1} & \dots & b_{4,m-3} \\
\vdots & \vdots & \ddots & \vdots \\
b_{m,0} & b_{m,1} & \dots & b_{m,m-3}
    \end{vmatrix}.\lb{wr3}                     
\end{align}  
Similar to \eqref{wr2},
none of the $3$ determinants on the right-handside are $0$.
Let
\begin{align*}
  f&=\dfrac{p_1^{k_1-m+3}}{\widetilde{w}}\begin{vmatrix}
b_{1,0} & b_{1,1} & \dots & b_{1,m-3} \\
b_{4,0} & b_{4,1} & \dots & b_{4,m-3} \\
\vdots & \vdots & \ddots & \vdots \\
b_{m,0} & b_{m,1} & \dots & b_{m,m-3}
                 \end{vmatrix}, \ \
  g=\dfrac{p_2^{k_2-m+3}}{\widetilde{w}}\begin{vmatrix}
b_{2,0} & b_{2,1} & \dots & b_{2,m-3} \\
b_{4,0} & b_{4,1} & \dots & b_{4,m-3} \\
\vdots & \vdots & \ddots & \vdots \\
b_{m,0} & b_{m,1} & \dots & b_{m,m-3}
    \end{vmatrix},\\
  h&=\dfrac{p_3^{k_3-m+3}}{\widetilde{w}}\begin{vmatrix}
b_{3,0} & b_{3,1} & \dots & b_{3,m-3} \\
b_{4,0} & b_{4,1} & \dots & b_{4,m-3} \\
\vdots & \vdots & \ddots & \vdots \\
b_{m,0} & b_{m,1} & \dots & b_{m,m-3}
    \end{vmatrix}.
\end{align*}  
Consequently, $f$, $g$ and $h$ are nonzero polynomials.
Let $d = \gcd(f,g,h)$. By
$\gcd\zj{p_1^{k_1-m+3}/\widetilde{w},\ 
  p_2^{k_2-m+3}/\widetilde{w},\ p_3^{k_3-m+3}/\widetilde{w}} = 1$
and \eqref{wr3}, we get that $d$ is a divisor of
the polynomial
\begin{align} H\stackrel{\textup{def}}{=}
  \begin{vmatrix}
b_{1,0} & b_{1,1} & \dots & b_{1,m-3} \\
b_{4,0} & b_{4,1} & \dots & b_{4,m-3} \\
\vdots & \vdots & \ddots & \vdots \\
b_{m,0} & b_{m,1} & \dots & b_{m,m-3}
  \end{vmatrix}\cdot \begin{vmatrix}
b_{2,0} & b_{2,1} & \dots & b_{2,m-3} \\
b_{4,0} & b_{4,1} & \dots & b_{4,m-3} \\
\vdots & \vdots & \ddots & \vdots \\
b_{m,0} & b_{m,1} & \dots & b_{m,m-3}
                     \end{vmatrix} \cdot
 \begin{vmatrix}
b_{3,0} & b_{3,1} & \dots & b_{3,m-3} \\
b_{4,0} & b_{4,1} & \dots & b_{4,m-3} \\
\vdots & \vdots & \ddots & \vdots \\
b_{m,0} & b_{m,1} & \dots & b_{m,m-3}
    \end{vmatrix}.\lb{wr4}                    
\end{align}
By \eqref{sr0}
\begin{align}
  \deg H&\le 3 \sum_{n=0}^{m-3} \max_i \{ \deg b_{i,n} \}\notag\\
  &= 3 \sum_{n=0}^{m-3} \max_i \{ (m-3) t_i -n\}\notag\\
  &< 3T \sum_{n=0}^{m-3} (m-3)\notag\\
  &= 3T(m-3)(m-2).\lb{HH} 
\end{align}
Let
\begin{align*}
f_1 = \frac{f}{d},\ \ g_1 = \frac{g}{d}, \ \ h_1 = \frac{h}{d}.
\end{align*}
Then, 
\[
  f_1 + g_1 + h_1 = 0.
\]
By \eqref{KK1}, \eqref{HH} and Lemma \ref{lemw} for the degree of the polynomial $f_1$ we have
\begin{align}
\text{deg } f_1 &= \text{deg } f - \text{deg } d\notag\\
&\ge \text{deg } f - \text{deg } H\notag\\
&\ge \text{deg } (p_1^{k_1-m+3})-\deg \widetilde{w}- \text{deg } H\notag\\
&\ge (K-m+3)T-\deg \widetilde{w} - 3(m-3)(m-2)T\notag\\
&\ge \zj{K-3m^2+14m-15-\deg \widetilde{w}/T}T\ge 1.\lb{mxdg}
\end{align}
Thus, $f_1'=g_1'=h_1'=0$ does not hold, hence we can apply Mason's theorem.
By Lemma \ref{lem5.2}, it follows that
\begin{align}
\max\{\deg f_1, \deg g_1, \deg h_1\} \le r^*(f_1g_1h_1) - 1\lb{Mas}
\end{align}
By \eqref{mxdg}
\[
  \max\{\deg f_1, \deg g_1, \deg h_1\}
  \ge \zj{K-3m^2+14m-15-\deg \widetilde{w}/T}T.
\]
On the other hand
\begin{align*}
  r^*(f_1g_1h_1)-1 &\le \text{deg }\zj{p_1p_2p_3H}-1\\
                   &< 3T+\text{deg }H\\
                   &< 3T+3T(m-3)(m-2)\\
                   &=3T\zj{m^2-5m+7}.  
\end{align*}
So, by \eqref{Mas}, we get
\begin{align*}
  \zj{K-3m^2+14m-15-\deg \widetilde{w}/T}T
  &< 3T\zj{m^2-5m+7}\\
  K &< 6m^2-29m+36+\deg \widetilde{w}/T.\\   
\end{align*}
By using Lemma \ref{lemw}, we get
\begin{align*}
  K< \dfrac{1}{3}\zj{m^3+7m^2-49m+68}. 
\end{align*}
In order to complete the proof of the theorem, we only need
to prove Lemma \ref{lemw}.

\bigskip\noindent\textbf{Proof of Lemma \ref{lemw}.}
To prove the lemma, let's introduce the following notation:
\[
\widetilde{P}_t \stackrel{\text{def}}{=} \text{gcd} \left( p_1^{k_1-m+t}, p_2^{k_2-m+t}, \dots, p_t^{k_t-m+t} \right)
\]
According to the theorem's assumption, $\widetilde{P}_m = 1$. However, $\widetilde{P}_{m-1}$ is also 1, since
\begin{align*}
\widetilde{P}_{m-1} &\mid p_1^{k_1-1}, p_2^{k_2-1}, \dots, p_{m-1}^{k_{m-1}-1}\\
\widetilde{P}_{m-1} &\mid p_1^{k_1}, p_2^{k_2}, \dots, p_{m-1}^{k_{m-1}}\\
\widetilde{P}_{m-1} &\mid p_1^{k_1} + \dots + p_{m-1}^{k_{m-1}} = -p_m^{k_m}\\
\widetilde{P}_{m-1} &\mid \text{gcd} \left( p_1^{k_1}, \dots, p_m^{k_m} \right) = 1.
\end{align*}
Thus, $\text{deg } \widetilde{P}_m = \text{deg } \widetilde{P}_{m-1} = 0$.

\bigskip
Next, we would like to provide an upper bound for
$\text{deg } \widetilde{P}_{t-1}$ 
using $\text{deg } \widetilde{P}_t$.
For this, let us consider the following Wronskian determinant:
\begin{align*}
W \stackrel{\text{def}}{=} \begin{vmatrix}
a_1+\dots+a_{m}
& a_1' +\dots a_{m}'&
 \dots & a_1^{(m-t)}+\dots +  a_{m}^{(m-t)}  \\
a_{t+1} & a_{t+1}' & \dots & a_{m}^{(m-t)} \\
a_{t+2} & a_{t+2}' & \dots & a_{m}^{(m-t)} \\
\vdots & \vdots & \ddots & \vdots \\
a_{m} & a_{m}' & \dots & a_m^{(m-t)} 
\end{vmatrix}=0.
\end{align*}
\bigskip
By subtracting the sum of the other rows from the first row, we get:
\begin{align}
W&= \begin{vmatrix}
      a_1+\dots+a_t
      & a_1'+\dots+a_t' & \dots & a_1^{(m-t)}+\dots+a_3^{(m-t)} \\
a_{t+1} & a_{t+1}' & \dots & a_{t+1}^{(m-t)} \\
a_{t+2} & a_{t+2}' & \dots & a_{t+2}^{(m-t)} \\
\vdots & \vdots & \ddots & \vdots \\
a_m & a_m' & \dots & a_m^{(m-t)}
\end{vmatrix}\notag\\
&= \footnotesize\begin{vmatrix}
a_1 & a_1' & \dots & a_1^{(m-t)} \\
a_{t+1} & a_{t+1}' & \dots & a_{t+1}^{(m-t)} \\
a_{t+2} & a_{t+2}' & \dots & a_{t+2}^{(m-t)} \\
\vdots & \vdots & \ddots & \vdots \\
a_m & a_m' & \dots & a_m^{(m-t)}
\end{vmatrix} + \begin{vmatrix}
a_2 & a_2' & \dots & a_2^{(m-t)} \\
a_{t+1} & a_{t+1}' & \dots & a_{t+1}^{(m-t)} \\
a_{t+2} & a_{t+2}' & \dots & a_{t+2}^{(m-t)} \\
\vdots & \vdots & \ddots & \vdots \\
a_m & a_m' & \dots & a_m^{(m-t)}
                \end{vmatrix} + \dots \notag\\
&+\begin{vmatrix}
a_t & a_t' & \dots & a_t^{(m-t)} \\
a_{t+1} & a_{t+1}' & \dots & a_{t+1}^{(m-t)} \\
a_{t+2} & a_{t+2}' & \dots & a_{t+2}^{(m-t)} \\
\vdots & \vdots & \ddots & \vdots \\
a_m & a_m' & \dots & a_m^{(m-t)}
                \end{vmatrix}.\lb{hhh}  
\end{align}
None of these $t$ determinants are $0$, since the
polynomials $a_i, a_{t+1}, a_{t+2}, \dots, a_m$ are linearly independent
if $1 \le i \le t$.
For $1\le i\le m-t$ and $0\le n\le k_i$ define $c_{i,n}\in\mathbb C[x]$
by
\[
  a_i^{(n)}=p_i^{k_i-m+t}c_{i,n}.
\]
Here
\begin{align}
  \textup{deg }c_{i,n} &= \textup{deg }a_i^{(n)}-\textup{deg }p_i^{k_i-m+t}
  \notag\\
                       &=(k_it_i-n)-(k_i-m+t)t_i\notag\\
                       &=(m-t)t_i-n\lb{ujabb}. 
\end{align}
From \eqref{hhh}, it follows that:
\begin{align}
&p_1^{k_1-m+t}
    \begin{vmatrix}
c_{1,0} & c_{1,1} & \dots & c_{1,m-t} \\
c_{t+1,0} & c_{t+1,1} & \dots & c_{t+1,m-t} \\
\vdots & \vdots & \ddots & \vdots \\
c_{m,0} & c_{m,1} & \dots & c_{m,m-t}
                 \end{vmatrix}
+ p_2^{k_2-m+t}
   \begin{vmatrix}
c_{2,0} & c_{2,1} & \dots & c_{2,m-t} \\
c_{t+1,0} & c_{t+1,1} & \dots & c_{t+1,m-t} \\
\vdots & \vdots & \ddots & \vdots \\
c_{m,0} & c_{m,1} & \dots & c_{m,m-t}
   \end{vmatrix}+\dots \notag\\
& + p_t^{k_t-m+t}
   \begin{vmatrix}
c_{t,0} & c_{t,1} & \dots & c_{t,m-t} \\
c_{t+1,0} & c_{t+1,1} & \dots & c_{t+1,m-t} \\
\vdots & \vdots & \ddots & \vdots \\
c_{m,0} & c_{m,1} & \dots & c_{m,m-t}
    \end{vmatrix}=0.\lb{wru3}                     
\end{align}  
Then, by \eqref{wru3}, $\widetilde{P}_{t-1}=\gcd \left(p_1^{k_1-m+t-1},\dots,p_{t-1}^{k_{t-1}-m+t-1}\right)$ is a divisor of
\begin{align}
  p_t^{k_t-m+t}
\begin{vmatrix}
c_{t,0} & c_{t,1} & \dots & c_{t,m-t} \\
c_{t+1,0} & c_{t+1,1} & \dots & c_{t+1,m-t} \\
\vdots & \vdots & \ddots & \vdots \\
c_{m,0} & c_{m,1} & \dots & c_{m,m-t}\lb{mat}
\end{vmatrix}. 
\end{align}
Let
\begin{align*}
  D&=\gcd (\widetilde{P}_{t-1},p_t^{k_1-m+t})\\
  \intertext{and}
  \widetilde{P}_{t-1} &= Dr\\
  p_t^{k_1-m+t}&=Ds.
\end{align*}  
Then,  $\gcd (r,s)=1$. So, by $\widetilde{P}_{t-1}$ is a divisor
of the expression in \eqref{mat}, we get $Dr$ is a divisor of 
\[
Ds  \begin{vmatrix}
c_{t,0} & c_{t,1} & \dots & c_{t,m-t} \\
c_{t+1,0} & c_{t+1,1} & \dots & c_{t+1,m-t} \\
\vdots & \vdots & \ddots & \vdots \\
c_{m,0} & c_{m,1} & \dots & c_{m,m-t}
\end{vmatrix}.
\]
Thus, $r$ is a divisor of
\[
  \begin{vmatrix}
c_{t,0} & c_{t,1} & \dots & c_{t,m-t} \\
c_{t+1,0} & c_{t+1,1} & \dots & c_{t+1,m-t} \\
\vdots & \vdots & \ddots & \vdots \\
c_{m,0} & c_{m,1} & \dots & c_{m,m-t}
\end{vmatrix}.
\]
But $D\mid p_1^{k_1-m+t-1},\dots,p_{t-1}^{k_{t-1}-m+t-1}$ and
$D\mid p_t^{k_t-m+t}$. Thus  $D\mid p_1^{k_1-m+t},\dots,p_{t}^{k_t-m+t}$,
from which $D\mid \widetilde{P}_{t}$. Then,
\[
\widetilde{P}_{t-1}=Dr\textup{ \ divides \ } \widetilde{P_t}   \begin{vmatrix}
c_{t,0} & c_{t,1} & \dots & c_{t,m-t} \\
c_{t+1,0} & c_{t+1,1} & \dots & c_{t+1,m-t} \\
\vdots & \vdots & \ddots & \vdots \\
c_{m,0} & c_{m,1} & \dots & c_{m,m-t}
\end{vmatrix}.
\]
By this and \eqref{ujabb},
\begin{align*}
  \deg \widetilde{P}_{t-1} &\le \deg \widetilde{P_{t}}
  +\deg  \begin{vmatrix}
c_{t,0} & c_{t,1} & \dots & c_{t,m-t} \\
c_{t+1,0} & c_{t+1,1} & \dots & c_{t+1,m-t} \\
\vdots & \vdots & \ddots & \vdots \\
c_{m,0} & c_{m,1} & \dots & c_{m,m-t}
         \end{vmatrix}\\
        &\le  \deg \widetilde{P_{t}} +\sum_{n=0}^{m-t}
          \max_i \{\deg c_{i,n}\}\\
&\le  \deg \widetilde{P_{t}} +\sum_{n=0}^{m-t}
                                    \max_i \{(m-t)t_i-n\}\\
&\le  \deg \widetilde{P_{t}} +T\sum_{n=0}^{m-t}(m-t)\\
&\le  \deg \widetilde{P_{t}} +(m-t)(m-t+1)T.
\end{align*}
By iterating this, we get:
\begin{align*}
\deg  \widetilde{P}_{3} &\le \deg  \widetilde{P}_{4} + (m-4)(m-3)T\\
&\le \deg  \widetilde{P}_{5} + (m-5)(m-4)T + (m-4)(m-3)T\\
&\ \ \vdots\\
&\le \deg  \widetilde{P}_{m-1} + \sum_{i=0}^{m-4} i(i+1)T\\
&= \sum_{i=0}^{m-4} (i^{2} + i)T\\
&= \frac{(m-4)(m-5)(2m-7)}{6}T + \frac{(m-4)(m-5)}{2}T\\
&= \frac{1}{3}\zj{m^{3} - 11m^{2} + 38m - 40}T.
\end{align*}
This completes the proof of the lemma and the theorem.

\bigskip\noindent\textbf{Proof of Corollary \ref{cor2.4}.}
Consider a linear combination of the polynomials $p_1^k, \dots,
p_m^k$ that equals $0$ and is as short as possible. Without loss of generality, this linear combination can be assumed to be in the form
\begin{align}
  \lambda_1 p_1^k + \dots + \lambda_s p_s^k = 0 \lb{cor1}
\end{align}
where $\lambda_i \in \mathbb{C} \setminus \{0\}$ and $3 \le s \le m$. Since $s$ is the shortest possible length under the given condition, any $s-1$ polynomials from $p_1^k, \dots, p_s^k$ are linearly independent. Let
\[
  d = \text{gcd} \{p_1, p_2, \dots, p_s\}.
\]
Define the polynomials $p_1',\dots,p_s'$ by the formula
\[
  p'_i = \lambda_i^{1/k} \frac{p_i}{d}.
\]
From \eqref{cor1}, it follows that
\[
  (p'_1)^k + \dots + (p'_s)^k = 0,
\]
where the polynomials $p'_1, \dots, p'_s$ satisfy the conditions of
Theorem \ref{thm2.3}, i.e., any $s-1$ of them are linearly independent and
$\gcd(p'_1, \dots, p'_s) = 1$. Applying Theorem \ref{thm2.3}, we get
\[
  k < \dfrac{1}{3}\zj{s^3+7s^2-49s+68}.
\]
Since $s \le m$, it follows that
\[
  k < \dfrac{1}{3}\zj{m^3+7m^2-49m+68}.
\]

\section{Remarks}

M. de Bondt \cite{Bondt} generalized Mason's theorem as follows:

\begin{thm}[de Bondt]\lb{thm6.1}
Suppose that  $m\ge 3$, for the (possibly multivariate) polynomials
$f_1,f_2,\dots,f_m$ over $\mathbb C$ we have
\[
f_1+f_2+\dots+f_m=0,
\]
and in case of $1\le i_1<i_2<\dots <i_s\le m$ we also have 
\[
f_{i_1}+f_{i_2}+\dots+f_{i_s}=0\ \ \ \Rightarrow \ \ \
\gc (f_{i_1},f_{i_2},\dots,f_{i_s})=1.
\]
Then,
\begin{equation}
\max_{1\le j\le m} \deg f_j
\le (m-2)(r^* (f_1)+r^* (f_2)+\dots +r^* (f_m)-1).
\lb{fuggvi}
\end{equation}
\end{thm}

\bs \textbf{From this general theorem, Theorem \ref{thm2.1} and \ref{thm2.3}
  follow.} I found the paper \cite{Bondt} exclusively on arXiv, and
it hasn’t been published yet and may be under review.
The proof presented here differs from de Bondt's proof and relies solely
on elementary arguments.

\bs In 1979, Newman and Slater wrote an excellent (and possibly the first)
paper on this topic. However, one of their
theorems (see \cite[p. 481]{NewmanSlater}) had missing conditions.
Specifically, 
they forgot to assume certain pairwise coprime
conditions for the polynomials involved, and the
condition $n\ge k^2-k$ is missing. In any case,
Newman and Slater's theorem remains a valuable
contribution to the field's development.

\bs Despite their theoretical importance, the theorems studied in this paper
are not completely useless, and I hope to write another paper about
their cryptographical application (e.g., the generation of a large family of binary sequences with small cross-correlation measures) in the future, but that's another story.

\bigskip

\textbf{Acknowledgement.}
I would like to thank Imre Ruzsa for his intriguing  lecture
on Mason's theorem. I am also grateful to the referee for their valuable
advice, which led to the establishment of Theorem \ref{thm2.3}.

\bigskip
\bigskip\noindent
\footnotesize{\textbf{Data Availability Statement:}
  All data generated or analysed during this study are included in the present
  paper.}

\bigskip\noindent\footnotesize{\textbf{Conflict of Interest Statement:}
The author declares that they have no conflict of interest.}

\noindent Eötvös Loránd University, Institute of Mathematics,\\
H-1117 Budapest, Pázmány Péter st. 1/C, Hungary\\
\noindent E-mail address: \texttt{katalin.gyarmati@gmail.com}


\normalsize

\begin{thebibliography}{99}

\bibitem{Bayat}M. Bayat and H. Teimoori, \textit{A generalization of Mason's theorem for four polynomials}, Elem. Math. 59(1) (2004), 23–28.  

\bibitem{Boc1}M. B\^ocher, \textit{The theory of linear dependence}, Ann. of Math. (2) 2 (1900/01), 81–96.

\bibitem{Boc2}M. B\^ocher, \textit{Certain cases in which the vanishing of the Wronskian is a suﬃcient condition for linear dependence}, Trans. Amer. Math. Soc. 2 (1901), 139–149.
 
\bibitem{Bondt}M. de Bondt, \textit{Another generalization of Mason's ABC-theorem}, arXiv:0707.0434, \url{https://arxiv.org/abs/0707.0434}.

\bibitem{Browkin}J. Browkin and J. Brzezi'nski, \textit{Some remarks on the abc-conjecture}. Math. Comp. 62 (206) (1994), 931–939.

\bibitem{Br}W. D. Brownawell and D. W. Masser, \textit{Vanishing sums in function fields}, Math. Proc. Cambridge Philos. Soc. 100 (1986), 427-434.

\bibitem{Elk}N. D. Elkies, \textit{Rational points near curves and small nonzero
    $|x^3-y^2|$ via lattice reduction},
  Algorithmic number theory (Leiden, 2000), 33–63. Lecture Notes in Comput. Sci. 1838, Springer-Verlag, Berlin, 2000.
  

\bibitem{GrTuck}A. Granville and T. J. Tucker, \textit{It’s as easy as abc}, Notices Am. Math. Soc. 49, No. 10, 1224-1231 (2002).

\bibitem{Hardy1}G. H. Hardy, \textit{Ramanujan}, Cambridge University Press, New York, 1940, p. 12.

\bibitem{Hardy2}G. H. Hardy, \textit{Srinivasa Ramanujan}, Proc. London Math. Soc., s2-19 (1) (1921), xl–lviii.

\bibitem{Hirs1} M. D. Hirschhorn, \textit{An amazing identity of Ramanujan}, Math. Mag. 68 (1995), 199–201.

\bibitem{Hirs2} M. D. Hirschhorn, \textit{A proof in the spirit of Zeilberger of an amazing identity of Ramanujan}, Math. Mag. 69 (1996), 267–269.

\bibitem{Hirs3} M.D. Hirschhorn, \textit{Ramanujan and Fermat's last Theorem}, Austral. Math. Soc. Gazette 31 (2004), 256-257.

\bibitem{Hu1}P.-C. Hu and C.-C. Yang, \textit{Notes on a generalized abc-conjecture over function fields}, Ann. Math. Blaise Pascal 8(1) (2001), 61-71.
 
\bibitem{Hu2}P.-C. Hu and C.-C. Yang, \textit{A generalized abc-conjecture over function fields}, J. Number Theory 94(2) (2002), 286-298.

\bibitem{Hur}W. Hurewicz, \textit{Lectures on Ordinary Diﬀerential Equations}, Technology Press of the Massachusetts Institute of Technology, Cambridge, MA, 1958.

\bibitem{Kru}M. Krusemeyer, \textit{The teaching of mathematics: Why does the Wronskian work?}, Am. Math. Mon. 95 (1) (1988), 46-49.

\bibitem{Lange} S. Lang, \textit{Die abc-Vermutung}, Elem. Math. 48(3), 89-99 (1993).
 
\bibitem{Lehmer}D. H. Lehmer, \textit{On the Diophantine equation x3+y3+z3=1}, Journal of the London Mathematical Society, Volume s1-31 (3) (1956), 275–280.

\bibitem{Masonfunction}R. C. Mason, \textit{Diophantic equations over function fields}, London Math. Soc. Lecture Note Series 96, Cambridge, 1984.
 
\bibitem{NewmanSlater}D. J. Newman and M. Slater, \textit{Waring's problem for the ring of polynomials}, J. Number Theory 11(4) (1979), 477–487.

\bibitem{Rama}S. Ramanujan, \textit{The lost notebook and other unpublished papers}, Springer-Verlag, Berlin; Narosa Publishing House, New Delhi, 1988, p. 341.

\bibitem{R}Ruzsa, I., \textit{Fermat, Beal, abc-sejt'es, Stothers-Mason-t'etel}, presentation in Hungarian, Alfr\'ed R\'enyi Institute of Mathematics, 2022.
 
\bibitem{ShapiroSparer}H. N. Shapiro and G.H. Sparer, \textit{Extensions of a theorem of Mason}, Comm. Pure. Appl. Math. XLVII (1994), 711-718.

\bibitem{Silverman}J.H. Silverman, \textit{The S-unit equation over function fields}, Math. Proc. Cambridge Philos. Soc. 95(2) (1984), 3–4.

\bibitem{bovites}StackExchange,
  \textit{Simple extension of $\mathbb Q(\sqrt[4]{2},i)$}.
  \url{https://math.stackexchange.com/questions/2221706/simple-extension-of-mathbbq-sqrt42-i}.
 
\bibitem{Eldisk}StackExchange, \textit{Solutions to $A^N+B^N=C^N\pm 1$, for
    $N\ge 4$}, \url{https://math.stackexchange.com/questions/4016969/solutions-to-anbn-cn-pm-1-for-n-geq-4}.

\bibitem{Syd}N. Snyder, \textit{An alternate proof of Mason's theorem}, Elemente der Mathematik, 55(3) (1984), 93–94.
  
\bibitem{Sto}W. W. Stothers, \textit{Polynomial identities and hauptmoduln}, Quarterly J. Math. Oxford, 2(32) (1981) 349–370.
 
\bibitem{TP}Turán P., \textit{Egy különös életút, Ramanujan}, in Hungarian [\textit{Intriguing life of Ramanujan}], can be found in ``Nagy Pillanatok a Matematika Történetében'' [Great Moments in the History of Mathematics], ed. R. Freud, Gondolat, Budapest, 1981.

\bibitem{Voloch}J. F. Voloch, \textit{Diagonal equations over function fields}, Bol. Soc. Brasil. Math. 16 (1985), 29-39.

\bibitem{Wrons}Wikipedia, \textit{Wronskian}, \url{https://en.wikipedia.org/wiki/Wronskian}.

\bibitem{Zannier}U. Zannier, \textit{Some remarks on the S-unit equation in function fields}, Acta Arith. 64 (1993), 87-98.

\end{thebibliography}
\end{document}